\documentclass[reqno]{amsart}

\usepackage[ngerman,english]{babel}
\usepackage{amssymb}
\usepackage{amsmath}
\usepackage{amsthm}
\usepackage[inline]{enumitem}
\usepackage{tikz}
\usepackage{mathtools}
\usepackage{hyperref}
\usepackage{cleveref}

\newtheorem{theorem}{Theorem}

\newtheorem{lemma}{Lemma}
\newtheorem*{lemma*}{Lemma}
\theoremstyle{definition}
\newtheorem{definition}{Definition}

\theoremstyle{remark}

\newcommand{\set}[2]{\{\, #1 \mid #2 \,\}}
\newcommand{\fg}{\mathrm{FG}}
\newcommand{\mg}{\mathbb{G}}
\newcommand{\mh}{\mathbb{H}}
\newcommand{\id}{\textrm{id}}
\newcommand{\src}{\mathbf{s}}
\newcommand{\trgt}{\mathbf{t}}

\tikzset{%
  dot/.style={circle,draw,fill=black,inner sep=0pt, minimum width=3pt},
  bigdot/.style={circle,draw,fill=black,inner sep=0pt, minimum width=5pt},
  shorter/.style={shorten >=1mm,shorten <=1mm},
  edge/.style={shorter,thick},
}

\title{Finite F-inverse covers do exist}
\author{Julian Bitterlich  \\
Technische Universit\"at Darmstadt
}
\date{\today}

\begin{document}

\begin{abstract}
We show that every finite inverse monoid
has an idempotent-separating cover
by a finite F-inverse monoid.
This provides a positive answer to a conjecture
of Henckell and Rhodes \cite{HenckellRhodes1991}.
\end{abstract}
\maketitle

\section{Intorduction}
An important theorem
in the area of inverse monoids
is \emph{McAlister's Covering Theorem} \cite{McAlister1974A,Lawson1998}.
\begin{theorem}\label{theorem of mcalister}
For every inverse monoid $M$ there is an
E-unitary inverse monoid $N$ and an
idempotent-separating
cover $\theta \colon N \to M$.
\end{theorem}
Typically, the statement of McAlister's Covering Theorem
is just given as: ``every inverse monoid has an E-unitary cover.''

The proof of McAlister's Covering Theorem can be adapted
to finite inverse monoids as well.
This finite version reads in short:
``every finite inverse monoid has a finite E-unitary cover.''
\begin{theorem}\label{finite macalister}
For every finite inverse monoid $M$ there is a
finite, E-unitary inverse monoid $N$ and an
idempotent-separating
cover $\theta \colon N \to M$.
\end{theorem}
Actually, in \cite{Lawson1998}
the classical version and the finite version
of McAlister's Covering Theorem
are proved as a single statement.

A strengthening of McAlister's Covering Theorem
is Lawson's Covering Theorem \cite{Lawson1998}.
It shows the existence of F-inverse covers instead
of `just' E-unitary covers.
\begin{theorem}\label{theorem of lawson}
For every inverse monoid $M$ there is an
F-inverse monoid $N$ and 
an idempotent-separating
cover $\theta \colon N \to M$.
\end{theorem}

In the sequel we show
that the finite version
of Lawson's Covering Theorem
is also valid.
\begin{theorem}\label{finite lawson}
For every finite inverse monoid $M$ there is a
finite F-inverse monoid $N$ and an idempotent-separating
cover $\theta \colon N \to M$.
\end{theorem}

\Cref{finite lawson}
was first conjectured
by Henckell and Rhodes \cite{HenckellRhodes1991}
as a possible route to an affirmative answer for the
pointlike conjecture, an important conjecture
in the theory of monoids (cf.\ \cite{HenckellMargolisPinRhodes1991}).
The `pointlike conjecture' was proved to be true by
Ash \cite{Ash1991} but the validity
of the finite version of Lawson's Covering Theorem
remained an open problem.
Some conditional results on the conjecture
of Henckell and Rhodes are given in
\cite{AuingerSzendrei2006,Szakacs2016,SzakacsSzendrei2016}.

The proof of \Cref{finite lawson} is divided in two parts.
In \Cref{section: covers} we introduce a compatibility
relation between groups and inverse monoids
and show how these compatible groups can be used
to obtain F-inverse covers (\Cref{p-product works if group is good} (ii)).
In \Cref{constructing strongly compatible groups},
which constitutes the technical part of this work,
we then show the existence of finite such compatible
groups (\Cref{G is an algebraic realisation}).
While developing these techniques we also give proofs
of McAlister's and Lawson's covering theorems.

In \Cref{section: inverse semigroups} we discuss
these covering theorems for inverse semigroups
and how they are implied by the inverse monoid versions.

The presentation of the
results here is self-contained.
We provide all definitions
about E-unitary and F-inverse monoids
in the next section.
For further information about this topic,
the reader may consult
the monograph of Lawson \cite{Lawson1998}.
\section{Inverse monoids}\label{inverse monoids}
\subsection*{Basic definitions}
An \emph{inverse monoid} is a monoid $M$ that has for
each element $x \in M$ a unique element $x^{-1} \in M$
such that $x = xx^{-1}x$ and $x^{-1} = x^{-1}xx^{-1}$.
The partial bijections on a set $X$
form the \emph{symmetric inverse monoid on $X$};
denoted by $I(X)$.
The Wagner--Preston Representation Theorem
tells us that every inverse monoid
is a submonoid of a symmetric inverse monoid.
Consequently, we can visualise many definitions
and results about inverse monoids in an illustrative way
using partial bijections.

The set of idempotents of an inverse monoid $M$ is denoted by $E(M)$.
In $I(X)$ the idempotents are just the restrictions
of the identity. So clearly, the idempotents of $I(X)$
commute and thus, by the Wagner--Preston Theorem, the idempotents of any inverse monoid do commute.

The \emph{natural partial order} on $M$
is defined by
\[ x \leq y \quad :\Longleftrightarrow \quad x = ey
\text{ for some $e \in E(M)$}.\]
In $I(X)$ the natural partial order
$x \leq y$ just amounts to `$x$ is a restriction of $y$'.
From this we can deduce, using Wagner--Preston, that the natural partial
order on an inverse submonoid $M$ of an inverse monoid $N$
is induced by the natural partial order on $N$.
Also note that the idempotents $E(M)$ are those elements that lie
beneath $1$ w.r.t. the natural partial order.

The \emph{minimum group congruence} on $M$
is defined by
\[ x \sigma y \quad :\Longleftrightarrow \quad
z \leq x,y \text{ for some $z \in M$}.\]
As the name suggest, $\sigma$ is the smallest
congruence $\rho$ on $M$ for which $M/\rho$ is a group.
Note that all idempotents lie in the same
$\sigma$-class and that $1$ is a maximal element
in this class w.r.t.\ the natural partial order.
\subsection*{E-unitary inverse monoids and F-inverse monoids}
A subset $X$ of an inverse monoid $M$
is called \emph{unitary} if for any elements
$x \in X$ and $z \in M$, $xz \in X$ or $zx \in X$
implies $z \in X$.

An inverse monoid is \emph{$E$-unitary}
if $E(M)$ is unitary.
It is easy to see that an inverse monoid is 
E-unitary if, and only if,
\[ e \leq x \text{ for some } e \in E(M) \quad
\Longrightarrow \quad x \in E(M)\]
for all $x \in M$.

An \emph{F-inverse monoid} is an inverse monoid
in which each $\sigma$-class
has a maximum
w.r.t.\ the natural partial order.
A inverse monoid is an F-inverse monoid if, and only if,
\[
z \leq x,y \text{ for some } z
\quad \Longrightarrow \quad 
x,y \leq z \text{ for some } z
\]
for all $x,y \in M$.
$F$-inverse monoids are also $E$-unitary:
if $e \leq x$ with $e \in E(M)$,
then $e \leq x,1$.
So there is a $z$ with $x,1 \leq z$.
Clearly, $z = 1$ and so $x \leq 1$.
Thus $x \in E(M)$.

\subsection*{Covers}
A \emph{homomorphism} between inverse monoids is a
semigroup homomorphism.
A homomorphisms $\theta \colon N \to M$ maps idempotents
to idempotents and is order-preserving
w.r.t.\ the natural partial order, i.e.,
\[ x \leq y \quad \Longrightarrow \quad \theta(x) \leq \theta(y).\]
A homomorphism is
a \emph{cover} if it is surjective;
it is \emph{idempotent-separating}
if it is injective on the idempotents.
\section{Constructing covers using compatible groups}\label{section: covers}
\subsection*{Inverse monoids with generators}
It is beneficial for us to use
inverse monoids with a fixed set of generators.
We let $P$ always stand for a set
of symbols with an associated involution $(\cdot)^{-1}$.
A \emph{$P$-generated inverse monoid}
is an inverse monoid $M$ with a set of generators
$\set{ p^M }{p \in P}$, s.t.
$(p^M)^{-1} = (p^{-1})^M$.

For $u = p_1 \dots p_n \in P^*$
we set $u^M = p_1^M \dots p_n^M$.
Note that $(\cdot)^M$ constitutes
a monoid homomorphism and 
that $(\cdot)^M$ is also compatible with inverses, i.e.,
$(u^{-1})^M = (u^M)^{-1}$
where $u^{-1} := p_n^{-1} \dots p_1^{-1}$.

Clearly any inverse monoid can be cast
as a $P$-generated inverse monoid
for a suitable choice of $P$,
and for finite inverse monoids
we can choose this $P$ to be finite
as well. We always assume
that $P$ is finite if we talk
about finite $P$-generated inverse monoids.

\subsection*{The product construction}
We describe a product
of $P$-generated inverse mon\-oids
and $P$-generated groups
that induces idempotent-separating covers.

For a $P$-generated monoid $M$ and a $P$-generated group $G$
let $M \times_P G$ be
the $P$-generated inverse monoid
given as a submonoid of $M \times G$
with generators $\set{ (p^M,p^G)}{p \in P}$.

The natural partial order on $M \times_P G$
is given by
\begin{align*} (m_1,g_1) \leq (m_2,g_2)
\quad &\Longleftrightarrow \quad
m_1 \leq m_2 \text{ and } g_1 = g_2,
\shortintertext{and the idempotents are given by}
(m,g) \in E(M \times_P G)
\quad &\Longleftrightarrow \quad
m \in E(M) \text{ and } g = 1.
\end{align*}
\begin{lemma}\label{p-product is idempotent-separating}
Let $\pi \colon M \times_P G \to M$ be the projection
to the first component.
Then $\pi \colon M \times_P G \to M$
is an idempotent-separating, surjective
homomorphism.
\end{lemma}
\begin{proof}
Clearly $\pi$ is a homomorphism
and surjective.
$\pi$ is injective
on $E(M \times_P G)$
since idempotents in $M \times_P G$
are purely characterised by their first
component.
\end{proof}
\subsection*{Compatible groups}
We introduce two
compatibility notions
between $P$-gener\-ated
inverse monoids
and $P$-generated groups that ensure
that the product $M \times_P G $
is (i) $E$-unitary or (ii) an $F$-inverse monoid.
\begin{definition}
A $P$-generated group $G$ is
\begin{enumerate}[label = (\roman*)]
\item \emph{compatible} with $M$ if for all $u \in P^*$
\[ u^G = 1 \quad \Longrightarrow \quad u^M \leq 1.\]
\item \emph{strongly compatible} with $M$ if
for all $u,w \in P^*$
\[ u^G = w^G
\; \Longrightarrow \;
v^G = u^G = w^G
\text{ and } u^M, w^M \leq v^M \text{ for some } v \in P^*.\]
\end{enumerate}
\end{definition}
\begin{lemma}\label{p-product works if group is good}
Let $M$ be a $P$-generated inverse monoid
and $G$ a $P$-generated group.
Then $ M \times_P G$
is 
\begin{enumerate}[label = (\roman*)]
\item E-unitary if $G$ is compatible with $M$.
\item an F-invers monoid if $G$ is strongly compatible with $M$.
\end{enumerate}
\end{lemma}
\begin{proof}
Let $N = M \times_P G$.
We use the fact that we can denote the elements in $N$
by $u^N = (u^M,u^G)$ for $u \in P^*$.
\begin{enumerate}[label = (\roman*)]
\item Let $u^N \leq w^N$
and $u^N \in E(N)$.
Then $u^G = w^G$ and $u^G = 1$.
Hence $w^G = 1$.
Since $G$ is compatible with $M$, this gives $w^M \leq 1$,
i.e., $w^M = e \in E(M)$.
So $w^N = (w^M,w^G) = (e,1) \in E(N)$.
\item Let $w^N \leq u_1^N,u_2^N$.
Then $u_1^G = w^G = u_2^G$.
Since $G$ is strongly compatible with $M$,
we obtain a $v \in P^*$ s.t.\
$v^G = u_1^G = u_2^G$ and $u_1^M,u_2^M \leq v^M$.
So $u_1^N = (u_1^M,u_1^G) \leq (v^M,v^G) = v^N$
and $u_2^N = (u_2^M,u_2wG) \leq (v^M,v^G) = v^N$.
\end{enumerate}
\end{proof}
Note that $\fg(P)$, the free group over $P$,
is strongly compatible with any $P$-generated inverse monoid.
With this we can prove the covering theorems
of McAlister and Lawson.
\begin{proof}[Proof of \Cref{theorem of mcalister} and \Cref{theorem of lawson}]
It suffices to show \Cref{theorem of lawson}
since F-inverse monoids are also E-unitary.

Let $M$ be an inverse monoid.
We can see $M$  as a $P$-generated inverses monoid
by a suitable choice of $P$ and generators of $M$.
Then $M \times_P \fg(P)$ is an F-inverse monoid
by \Cref{p-product works if group is good}.
Furthermore, $\pi \colon M \times_P \fg(P) \to M$
is an idempotent-separating cover
by \Cref{p-product is idempotent-separating}.
\end{proof}
We now want to prove the finite
versions of these covering theorems.
For that we have to provide finite (strongly) compatible groups.
\begin{lemma}\label{finite compatible groups}
For every finite $P$-generated inverse monoid
there is a finite compatible $P$-generated group.
\end{lemma}
\begin{proof}
By the Wagner--Preston Theorem
we can think of the $P$-generated inverse monoid $M$
as a submonoid of $I(X)$ for some finite set $X$.
Then every $p^M$ is a partial bijection on $X$.
We can extend each of the $p^M$ to
a bijection $p^G$.
Then the subgroup of the symmetric group of $X$
generated by the $\set{ p^G }{p \in P}$
is compatible with $M$.
\end{proof}
We can now finish the proof of the finite
version of McAlister's Covering Theorem.
\begin{proof}[Proof of \Cref{finite macalister}]
Let $M$ be a finite inverse monoid.
We can see $M$ as a finite $P$-generated inverse monoid
by a suitable choice of $P$ and generators of $M$.
\Cref{finite compatible groups} guarantees
the existence of a finite $P$-generated group $G$ compatible with $M$.
The product $M \times_P G$ is a finite, E-unitary inverse monoid
by \Cref{p-product works if group is good}.
Furthermore, $\pi \colon M \times_P G \to M$
is an idempotent-separating cover
by \Cref{p-product is idempotent-separating}.
\end{proof}
Similarly, with the help of the next lemma
we can obtain the finite version
of Lawson's Covering Theorem.
\begin{lemma}\label{finite strongly compatible groups}
For every finite $P$-generated inverse monoid
there is a finite strongly compatible $P$-generated group.
\end{lemma}
\begin{proof}
The proof is the subject of \Cref{constructing strongly compatible groups}.
There we describe a construction of a
finite $P$-generated group $G$ for a given finite $P$-generated
inverse monoid $M$ that is compatible with $M$ by \Cref{G is an algebraic realisation}. 
\end{proof}
\begin{proof}[Proof of \Cref{finite lawson}]
Let $M$ be a finite inverse monoid.
We can see $M$ as a finite $P$-generated inverse monoid
by a suitable choice of $P$ and generators of $M$.
\Cref{finite strongly compatible groups} guarantees
the existence of a finite $P$-generated group $G$ strongly compatible with $M$.
The product $M \times_P G$ is a finite, F-inverse monoid
by \Cref{p-product works if group is good}.
Furthermore, $\pi \colon M \times_P G \to M$
is an idempotent-separating cover
by \Cref{p-product is idempotent-separating}.
\end{proof}
\section{Covering theorems for inverse semigroups}
\label{section: inverse semigroups}
The covering theorems
of McAlister and Lawson
are originally stated for for inverse semigroups.
In this section we want to argue
that these distinctions in the covering theorems
do not matter as we can deduce
the inverse semigroup versions from the inverse monoid versions
and vice versa.

An \emph{inverse semigroup}
is a semigroup that has unique inverses
as defined for inverse monoids.
All notions introduced so far translate 
to inverse semigroups as well.
(F-inverse semigroups are an exception; we discuss
this after the following lemma.)

It is easy to see that the covering theorems
for inverse semigroups directly imply
their counterparts for inverse monoids
by virtue of the following lemma.
\begin{lemma*}
Let $\theta \colon S \to M$ be a surjective, idempotent-separating
homomorphism of inverse semigroups.
Then $S$ is a monoid if $M$ is a monoid.
\end{lemma*}
\begin{proof}
Let $f \in S$ with $\theta(f) = 1$.
W.l.o.g. $f \in E(S)$,
otherwise we continue with $ff^{-1}$.
For $e \in E(S)$, $\theta(fe) = \theta(e)$
and so $fe = e$ as $\theta$ is injective on $E(S)$.
Thus $fe = e$ for all $e \in E(S)$.
Now, for $x \in S$,
\[ fx = fxx^{-1}x =  (fxx^{-1})x = xx^{-1}x = x.\]
Similarily $xf = x$. 
Thus $f$ is a neutral element in $S$.
\end{proof}

For F-inverse covers the situation
is actually a bit more complicated.
If we transfer the definition
of F-inverse monoids
to inverse semigroups we get
that each such `F-inverse semigroup'
is automatically a monoid.
However, there is a definition
of F-inverse semigroups that extends
the notion of F-inverse monoids
but does not force a semigroup
to be a monoid (see \cite[Chapter 7.4]{Lawson1998}
for the definition and properties of F-inverse semigroups).
Lawson's Covering Theorem
is given with this definition
of F-inverse semigroups in mind.
Nevertheless, the notions of F-inverse semigroups
and F-inverse monoids agree on the class of monoids
and so
Lawson's Covering Theorem for inverse semigroups
directly implies Lawson's Covering Theorem for inverse monoids.

Now we describe how to obtain the
covering theorems for inverse semigroups
using the covering theorems for inverse monoids.
Let $S$ be an inverse semigroup.
Then $S^1$ is constructed from 
$S$ by adding a neutral element.
Note that $S^1$ does not have any
non-trivial units.
The idea is now that we can
apply the covering theorems
to $S^1$ and subsequently 
remove the $1$ again.

The next lemma shows
that we can get McAlister's Covering Theorem for inverse semigroups
from the corresponding theorem for inverse monoids.
\begin{lemma*}
Let $N$ be an E-unitary inverse monoid, $S$ an inverse semigroup
and $\theta \colon N \to S^1$ an idempotent-separating cover.
Then $N\setminus \ker(\theta)$ is an E-unitary inverse semigroup
and
$\theta|_{N\setminus \ker(\theta)} \colon N\setminus \ker(\theta) \to S$
is an idempotent-separating cover.
\end{lemma*}
\begin{proof}
Clearly, $N\setminus \ker(\theta)$ is closed under inverses.
It is also closed under products
since $S^1$ does not have any non-trivial units.
So $N\setminus \ker(\theta)$ is an inverse subsemigroup of $N$.
Being E-unitary is a universal statement and hence
it is preserved under passage to inverse subsemigroups.
So $N \setminus \ker(\theta)$ is $E$-unitary.
It is clear that $\theta|_{N\setminus \ker(\theta)} \colon N\setminus \ker(\theta) \to S$
is an idempotent separating cover.
\end{proof}
We want to prove a similar statement
for Lawson's Covering Theorem.
For that we need the following result \cite[Chapter 7.4 Lemma 8]{Lawson1998}
\begin{lemma*}
Let $S$ be an inverse subsemigroup
of an F-inverse semigroup $T$ s.t.
\begin{enumerate}[label = (\roman*)]
\item $E(S)$ is an order ideal of $E(T)$, i.e.,
\[ f \leq e \quad \Longrightarrow \quad f \in E(S)\]
for $e \in E(S)$ and $f \in E(T)$,
\item for $t \in T$,
\[t^{-1}t,tt^{-1} \in S \quad \Longrightarrow \quad t \in S.\]
\end{enumerate}
Then $S$ is also an F-inverse semigroup.
\end{lemma*}
\begin{lemma*}
Let $N$ be an F-inverse monoid, $S$ an inverse semigroup
and $\theta \colon N \to S^1$ an idempotent-separating cover.
Then $N\setminus \ker(\theta)$ is an F-inverse semigroup
and
$\theta|_{N\setminus \ker(\theta)} \colon N\setminus \ker(\theta) \to S$
is an idempotent separating cover.
\end{lemma*}
\begin{proof}
The only non-trivial part is to check that
$N\setminus \ker(\theta)$ satisfies the properties (i) and (ii)
of the previous lemma.

For (i) let $e,f \in E(N)$ with $f \leq e$ and  $\theta(e) \neq 1$.
Then $\theta(f) \leq \theta(e) \neq 1$ and so $\theta(f) \neq 1$.
The right-hand side of (ii) translates to `$\theta(t)$ is a unit'
but the only unity in $S^1$ is $1$. So (ii) is satisfied as well.
\end{proof}
\section{Constructing strongly compatible groups}\label{constructing strongly compatible groups}
This section is solely dedicated to the proof
of \Cref{finite strongly compatible groups}.
We describe a construction that for any given
$P$-generated inverse monoid $M$ produces
a $P$-generated group strongly compatible with $M$.
A key part (or one could say black box)
in this construction is a theorem by Otto
(\Cref{ottos theorem}).

Before we can describe and discuss the construction
we need to introduce some notations
regarding finitely generated groupoids,
in order to state Otto's Theorem.

\subsection*{Finitely generated groupoids and the Theorem of Otto}
We see a groupoid $\mg$ as a generalisation of a group
in which every element $g \in \mg$ has an associated source,
$\src(g)$, and a target, $\trgt(g)$,
which impose the usual restrictions on the multiplication operation.
The sources and targets of $\mg$ constitute
the objects of the groupoid.
We denote the neutral element at an object $a$ by $\id_a$.

Similarly to $P$-generated groups we want
to work with groupoids with generators.
By their typed nature it is natural to use graphs
to describe the generators.

A multidigraph $I = (V,E)$ is a two-sorted
structure with  vertices $V$
and edges $E$.
Every edge $e \in E$ has a source, $\src(e) \in V$,
and target, $\trgt(e) \in V$.
We also assume that there is an involution
$(\cdot)^{-1}$ on $E$ s.t.\ $\src(e^{-1}) = \trgt(e)$
and $\trgt(e^{-1}) = \src(e)$.
A walk $u$ in $I$ is a sequence
of edges $u = e_1\dots e_n$
s.t.\ $\trgt(e_i) = \src(e_{i+1})$.
We denote all walks over $I$
by $I^*$.
For $\alpha \subseteq E$ closed under $(\cdot)^{-1}$
we let $I(\alpha)$ be the subgraph $(V,\alpha)$.
So $I(\alpha)^*$ denotes all walks in $I$
which only consist of edges in $\alpha$.

An $I$-groupoid is a groupoid $\mg$ with generators
$\set{ e^\mg }{e \in E}$ s.t.\
\begin{align*}
\src(e^\mg) = \src(e), \trgt(e^\mg) = \trgt(e)\text{, and } (e^{-1})^\mg = (e^\mg)^{-1}.
\end{align*}
We let $u^\mg := e_1^\mg \dots e_n^\mg$
for $u = e_1\dots e_n \in E^*$,
and $\mg(\alpha) := \set{ u^\mg }{ u \in I(\alpha)^*}$
for $\alpha \subseteq E$ closed under $(\cdot)^{-1}$.

$\mg$ is \emph{$2$-acyclic} if for all
$\alpha,\beta \subseteq E$ closed under $(\cdot)^{-1}$
\[ \mg(\alpha) \cap \mg(\beta) = \mg(\alpha \cap \beta).\]
See \Cref{fig: sketch of 2-cycle}
for a sketch that depicts which forms of cyclic configurations 
are forbidden by $2$-acyclicity.
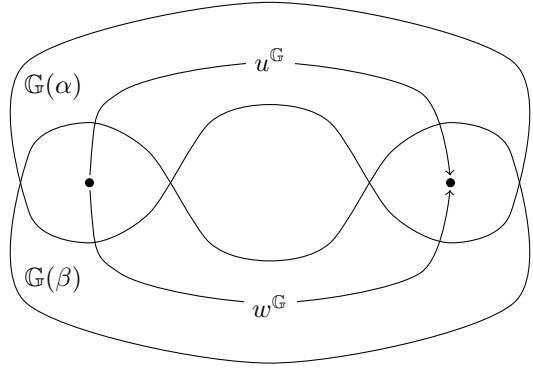
\begin{figure}
\begin{tikzpicture}[scale = 0.4,dot/.style={circle,draw,fill=black,inner sep=0pt, minimum width=3pt},shorter/.style={shorten >=1mm,shorten <=1mm},]
\draw [yshift=-1cm]  plot [smooth cycle] coordinates {(0,3.6) 
(2,3)  (4,0)  (6,-1)  (8,0)  (8.1,5)
(0,7)
(-8.1,5)  (-8,0)  (-6,-1)  (-4,0)  (-2,3)};
\draw [yshift=1cm] plot [smooth cycle] coordinates {(0,-3.6) 
(2,-3)  (4,0)  (6,1)  (8,0)  (8.1,-5)
(0,-7)
(-8.1,-5)  (-8,0)  (-6,1)  (-4,0)  (-2,-3)};
\node[dot] (a) at (-6,0) {};
\node[dot] (b) at (6,0) {};
\draw [->, shorter] plot [smooth] coordinates {
(a) (-5,3)  (0,4)  (5,3)  (b)};
\draw [->, shorter] plot [smooth] coordinates {
(a) (-5,-3)  (0,-4)  (5,-3)  (b)};
\node at (-7.2,3.2) {$\mg(\alpha)$};
\node[fill=white] at (0,4.1) {$u^\mg$};
\node at (-7.2,-3.2) {$\mg(\beta)$};
\node[fill=white] at (0,-4.1) {$w^\mg$};
\end{tikzpicture}
\caption{Sketch of a $2$-cycle.
There are $u \in I(\alpha)^*$ and $w \in I(\alpha)^*$
with $u^\mg = w^\mg$ but
there is no $v \in I(\alpha \cap \beta)^*$
with $v^\mg = u^\mg = w^\mg$.
}\label{fig: sketch of 2-cycle}
\end{figure}

A symmetry of $I$ is 
a two sorted map $\varphi = (\varphi_V,\varphi_E)$ s.t.
\[ \src(\varphi_E(e)) = \varphi_V(\src(e)),
\trgt(\varphi_E(e)) = \varphi_V(\trgt(e))
\text{, and }
\varphi_E(e^{-1})  = \varphi_E(e)^{-1}.\]
An $I$-groupoid $\mg$ is \emph{symmetric} if every
symmetry $\varphi$ of $I$ induces an automorphism $\varphi_\mg$
of $\mg$ induced by $\varphi_{\mg}(e^\mg) = \varphi_E(e)^\mg$.
The following theorem is due to Otto \cite{Otto2013}
(also see \cite{Otto2015} for an extended version).
\begin{theorem}\label{ottos theorem}
For every finite $I = (V,E)$ 
there are finite, symmetric, $2$-acyclic $I$-groupoids. 
\end{theorem}
\subsection*{The main construction}
We describe a construction of a finite
$P$-generated group $G$ that is
strongly compatible
with a given finite, $P$-generated inverse monoid $M$.
The construction proceeds in $4$ `steps':
\begin{enumerate}[label = Step \arabic*:]
\item Let $F$ be a finite $P$-generated group
compatible with $M$ (guaranteed to exist by \Cref{finite compatible groups}).
\item Let $I$ be the Gaifman graph of $F$, i.e.,
the finite multidigraph $I = (F,E)$ where
\begin{align*}
&E = \set{(f,p)}{ f \in F, p \in P} \text{ and }\\
&\src((f,p)) = f,\; \trgt((f,p)) = f p^F,\;
(f,p)^{-1} = (fp^F,p^{-1}).
\end{align*}.
\item Let $\mh$ be a finite, symmetric, $2$-acyclic $I$-groupoid
(guaranteed to exist by Otto's Theorem (\Cref{ottos theorem})).
\item Let $G$ be the $P$-generated group given
as a subgroup of the symmetric group of $\mh$
(here $\mh$ seen as a plain set)
generated by $\set{p^G}{p \in P}$
where 
\[p^G(h) := h(\trgt(h),p)^\mh.\]
\end{enumerate}
Clearly the resulting $P$-generated group $G$ is finite.
It remains to show that $G$ is indeed strongly compatible
with $M$.

In the following we reserve $a,b,c$ to denote elements in $P^*$
and $u,v,w$ to denote elements in $I^*$.

Note that we can think of $I$ as a cover of $P$,
each edge $(f,p)$ in $I$ is canonically labelled by $p \in P$
and every vertex in $I$ is adjacent to exactly one
edge with colour $p$.
This enables us to pass from $I^*$ to $P^*$
by projecting walks $(f_1,p_1)\dots(f_n,p_n)$ to
words $p_1\dots p_n$; we write $\pi$ for this projection.
On the other hand, for a fixed $f \in F$, we can
uniquely lift a word $a$ over $P$ to a walk $u$ in $I$ starting at $f$,
i.e., there is a unique $u \in I^*$ s.t.\ $\src(u) = f$
and $\pi(u) = a$.
So, lifts and projections give us means to translate between $P^*$ and $I^*$.

The following lemma shows that $G$ is strongly compatible
with $M$.
The lemma references some auxiliary statements
that are proved subsequently.
It might be instructive to read the proof
of the lemma once to get an idea about the crucial
steps, then read the proof of the auxiliary lemmas,
and after that read this proof once more.
\begin{lemma}\label{G is an algebraic realisation}
Let $M$ be a finite $P$-generated inverse monoid
and $G$ as above.
Then $G$ is a finite $P$-generated group
strongly compatible with $M$.
\end{lemma}
\begin{proof}
Clearly, $G$ is finite.
We show that $G$ is strongly compatible with $M$.
Let $a,b \in P^*$ with $a^G = b^G$.
We set $u, w \in I^*$ to be the lifts of $a,b$ to $1$.
Then, by \Cref{true facts about G} (i),
$u^\mh = a^G(\id_1)$ and $v^\mh = b^G(\id_1)$.
Thus $u^\mh = v^\mh$.

Set $\alpha = \set{ e,e^{-1} \in E}{\text{$e$ appears in $u$}}$
and $\beta = \set{ e,e^{-1} \in E}{\text{$e$ appears in $w$}}$.
Then $u \in I(\alpha)$, $w \in I(\beta)$,
and $u^\mh = w^\mh$.
By $2$-acyclicity of $\mh$, there is a $v \in I(\alpha \cap \beta)$
s.t.\ $v^\mh = u^\mh = w^\mh$.

We show that $c := \pi(v)$ is as desired, i.e.,
$c^G = a^G = b^G$ and $a^M,b^M \leq c^M$.
Note that $v$ is the lift of $c$ to $1$
and so, again by \Cref{true facts about G} (i),
$v^\mh = c^G(\id_1)$.
As $a^G,b^G$ and $c^G$ agree for one argument, namely $\id_1$,
they have to be equal, according to \Cref{true facts about G} (ii).
By construction of $v$, every edge $e$ that appears
in $v$ also appears in $u$ and in $w$ either
directly as $e$ or as $e^{-1}$. So by \Cref{true facts about I}
we get that $\pi(u)^M,\pi(w)^M \leq \pi(v)^M$
and thus $a^M,b^M \leq c^M$.
\end{proof}

We give now the proofs of the auxiliary lemmas.
\begin{lemma}\label{true facts about I}
Let $u,v \in I^*$ with the same sources
and targets
s.t.\ every edge or its inverse that appears in $v$ also 
appears in $u$.
Then $\pi(u)^M \leq \pi(v)^M$.
\end{lemma}
\begin{proof}
Note that
$\pi(w)^M \in E(M)$ if $\src(w) = \trgt(w)$:
it can be shown by induction that
$\pi(w)^F = \src(w)^{-1}\trgt(w)$ for all $w \in I^*$
(keep in mind that the vertices of $I$ are the elements of $F$).
If now $\src(w) = \trgt(w)$, then
$\pi(w)^F = 1$ and thus $\pi(w)^M \in E(M)$
as $F$ is compatible with $M$

We prove the statement of the lemma by induction
over the length of $v$.
If $|v| = 0$, then $\src(u) = \trgt(u)$ and thus
$\pi(u)^M \leq 1 = \pi(v)^M$.
For the induction step let $e$ be the last edge in $v$,
i.e., $v = v'e$.
Then $e$ or $e^{-1}$ also appears somewhere in $u$.
We distinguish these two cases.
\begin{enumerate}[label = \arabic*.:]
\item $e$ appears in $u$. Then $u$ can be decomposed
into $u = u_1eu_2$ as in the following sketch
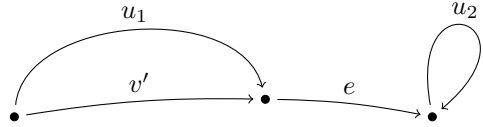
\begin{center}
\begin{tikzpicture}[scale = 0.4,dot/.style={circle,draw,fill=black,inner sep=0pt, minimum width=3pt},shorter/.style={shorten >=1mm,shorten <=1mm},]
\def\diam{40cm}
\def\bendiness{5}
\def\bendinessB{80}

\node[dot] (0) at (100:\diam) {};
\node[dot] (1) at (88:\diam) {};
\node[dot] (2) at (80:\diam) {};

\draw (0) edge[->, shorter,bend left=\bendiness] node[above] {$v'$} (1);
\draw (1) edge[->, shorter,bend left=\bendiness] node[above] {$e$}  (2);

\draw (0) edge[->, shorter,bend left=\bendinessB] node[above] {$u_1$} (1);
\draw (2) edge[->, shorter, out=100,in=40,looseness = 80]
node[above] {$u_2$} (2);
\end{tikzpicture}
\end{center}
We see that $\src(u_2) = \trgt(u_2)$ and so $\pi(u_2)^M \in E(M)$.
Note that we can apply the induction hypothesis
to $u_1eu_2u_2^{-1}e^{-1}$ and $v'$, i.e.,
$\pi(u_1eu_2u_2^{-1}e^{-1})^M \leq \pi(v')^M$.
So we get
\begin{align*}
\pi(u_1eu_2)^M &= \pi(u_1)^M\pi(eu_2)^M
= \pi(u_1)^M\pi(eu_2)^M(\pi(eu_2)^M)^{-1}\pi(eu_2)^M\\
&= \pi(u_1eu_2u_2^{-1}e^{-1})^M\pi(eu_2)^M
\leq \pi(v')^M\pi(eu_2)^M\\
&= \pi(v'e)^M\pi(u_2)^M \leq \pi(v'e)
\end{align*}
Note that the last inequality is true
just by the definition of $\leq$.
\item $e^{-1}$ appears in $u$. Then $u$ can be decomposed
into $u = u_1e^{-1}u_2$. A sketch of how $u$ and $v$ decompose
is given here:
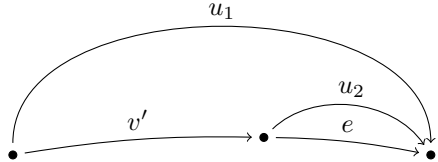
\begin{center}
\begin{tikzpicture}[scale = 0.4,dot/.style={circle,draw,fill=black,inner sep=0pt, minimum width=3pt},shorter/.style={shorten >=1mm,shorten <=1mm},]
\def\diam{40cm}
\def\bendiness{5}
\def\bendinessB{90}
\def\bendinessC{50}

\node[dot] (0) at (100:\diam) {};
\node[dot] (1) at (88:\diam) {};
\node[dot] (2) at (80:\diam) {};

\draw (0) edge[->, shorter,bend left=\bendiness] node[above] {$v'$} (1);
\draw (1) edge[->, shorter,bend left=\bendiness] node[above] {$e$}  (2);

\draw (0) edge[->, shorter,bend left=\bendinessB] node[above] {$u_1$} (2);
\draw (1) edge[->, shorter,bend left=\bendinessC] node[above] {$u_2$} (2);
\end{tikzpicture}
\end{center}
We see that $\src(e^{-1}u_2) = \trgt(e^{-1}u_2)$ and
so $\pi(e^{-1}u_2)^M \in E(M)$. Also note that
we can apply the induction hypothesis to $u_1e^{-1}u_2u_2^{-1}$
and $v'$, i.e., $\pi(u_1e^{-1}u_2u_2^{-1})^M \leq \pi(v')^M$.
So we get
\begin{align*}
\pi(u_1e^{-1}u_2)^M &= \pi(u_1)^M\pi(e^{-1}u_2)^M =
\pi(u_1)^M\pi(e^{-1}u_2)^M\pi(e^{-1}u_2)^M\\
&= \pi(u_1)^M\pi(e^{-1}u_2)^M\pi(u_2^{-1}e)^M
= \pi(u_1e^{-1}u_2u_2^{-1})^M\pi(e)^M\\
&\leq \pi(v')^M\pi(e)^M = \pi(v'e)^M.
\end{align*}
\end{enumerate}
In both cases $\pi(u)^M \leq \pi(v)^M$.
\end{proof}

We now want to show that the elements
in $G$ are completely determined
by one pair of value and argument.
This fact uses that $\mh$
is symmetric.

Every $f \in F$ defines a symmetry $\phi_f$ of $I$
whose action on the vertices is given by
$\phi_f(f') = ff'$ and its action on
the edges given by $\phi_f((f',p)) = (ff',p)$
(notationally we do not explicitly distinguish between
the vertex part and the edge part of $\phi$).
Note that $\phi_f((f',p)) = (\phi_f(f'),p)$.
Since $\mh$ is symmetric, $\phi_{f}$ can be extended to a
symmetry $\phi_{f,\mh}$ of $\mh$,
i.e., $\phi_{f,\mh}(u^\mh) := (\phi_f(u))^\mh$
is well-defined.
We give a proof sketch of the following fact:
\begin{align} \phi_{f,\mh}(a^G(h)) = a^G(\phi_{f,\mh}(h))
\tag{$*$}\end{align}
for $a \in P^*$ and $h \in \mh$.
The proof is by induction on the length of $a$,
and at its core lies the statement that that
$\phi_{f,\mh}(p^G(u^\mh)) = a^G(\phi_{f,\mh}(u^\mh))$
for $p \in P$ and $u \in I$.
We show this by proving that both sides are equal to $\phi_f\big(u (\trgt(u),p)\big)^\mh$.
For the left-hand side we get
\begin{align*}
\phi_{f,\mh}\big(p^G(u^\mh)\big) =
\phi_{f,\mh}\big(u^\mh (\trgt(u^\mh),p)^\mh\big)
= \phi_{f,\mh}\big(u^\mh (\trgt(u),p)^\mh\big)
= \phi_f\big(u(\trgt(u),p)\big)^\mh
\end{align*}
and for the right-hand side we get the same
\begin{align*}
p^G\big(\phi_{f,\mh}(u^\mh))
&= p^G\big( \phi_f(u)^\mh \big)
= \phi_f(u)^\mh (\trgt(\phi_f(u)^\mh),p)^\mh
= \phi_f(u)^\mh (\trgt(\phi_f(u)),p)^\mh\\
&= \phi_f(u)^\mh (\phi_f(\trgt(u)),p)^\mh
= \phi_f(u)^\mh \phi_f((\trgt(u),p))^\mh
= \phi_f\big( u(\trgt(u),p) \big)^\mh
\end{align*}

\begin{lemma}\label{true facts about G}
Let $a,b \in P^*$.
Then
\begin{enumerate}[label = (\roman*)]
\item $a^G(\id_1) = u^\mh$, where $u$ is the lift of $a$ to $1$ in $I$,
\item $a^G = b^G$ if $a^G(h) = b^G(h)$ for some $h \in \mh$.
\end{enumerate}
\end{lemma}
\begin{proof}
(i) is a consequence of the stronger statement
\[u^\mh = \pi(u)^G(\id_{\src(u)}),\]
for $u \in I^*$. This stronger
statement can be easily proved by induction.

To prove (ii) we note that in general
\begin{align} a^G(h') = h'\cdot h^{-1}\cdot a^G(h)
\;\text{ for all } h,h' \text{ with } \trgt(h) = \trgt(h'),
\tag{$**$}\end{align}
which can be proved easily by induction.
With this we can show that for
$h,h' \in \mh$, $a \in P^*$ and 
$\eta = \phi_{\trgt(h')\trgt(h)^{-1},\mh}$
we have that
\begin{align*}
a^G(h') &\overset{(**)}{=} h' \cdot \id_{\trgt(h')}^{-1} \cdot a^G(\id_{\trgt(h')})
= h' \cdot  a^G(\eta(\id_{\trgt(h)}))\\
&\overset{(*)}{=} h' \cdot  \eta(a^G(\id_{\trgt(h)}))
\overset{(**)}{=} h' \cdot \eta(\id_{\trgt(h)}\cdot h^{-1} \cdot a^G(h))\\
&= h' \cdot \eta(h^{-1} \cdot a^G(h)).
\end{align*}
We can now finish the argument for (ii).
If $a^G(h) = b^G(h)$,
then
$a^G(h') = h'  \eta(h^{-1}  a^G(h)) = 
h'  \eta(h^{-1}  b^G(h)) = b^G(h')$
for every $h' \in \mg$.
\end{proof}

\bibliographystyle{abbrv}
\bibliography{diss.bib}{}

\end{document}